\documentclass[11pt]{article}
\usepackage{graphicx}
 \usepackage[reqno, namelimits, sumlimits]{amsmath}
   \usepackage{amsfonts}
   \usepackage{amssymb}
  \usepackage{amsthm}
   \numberwithin{equation}{section}

\newtheorem{thm}{Theorem}[section]
\newtheorem{corollary}[thm]{Corollary}
\newtheorem{lemma}[thm]{Lemma}
\newtheorem{proposition}[thm]{Proposition}

\newtheorem{remark}[thm]{Remark}

\def\R{{\mathbb R}}

\def\N{{\mathbb N}}

\font\textmsbm=msbm10
\newfam\Bbbfam
\textfont\Bbbfam=\textmsbm

\date{ }
\begin{document}
\setcounter{page}{0}
\title{Uniform oscillation properties of almost periodic functions.}

\maketitle

\medskip\author{ \center  Alain HARAUX \\Sorbonne Université, Université Paris Cité, \\ CNRS, Laboratoire Jacques-Louis Lions, \\LJLL, F-75005 Paris, France.\\ email: alain.haraux@sorbonne-universite.fr
 \\}

\bigskip

\renewcommand{\abstractname} {\bf Abstract}
\begin{abstract} A synthetic view is provided on scattered results which were never presented in this form in the literature. First we  investigate precisely the oscillation length for periodic functions. Then, we establish the existence of a uniform oscillation lentgh for any scalar almost periodic function with mean-value 0. After that, we investigate upper bounds of the oscillation length for some almost periodic functions which appear in the study of continuum mechanics, and we show the existence of a finite oscillation length with respect to time in any open subset of some vibrating continuous media. The paper finally recalls a long standing open problem on the vibrations of a square membrane with fixed edge.

\end{abstract}

\bigskip

\qquad {\small \bf AMS classification numbers:}  34C10, 35B05, 35B15, 42A55, 42A75, 43A99
 
\vfill

\newpage

\noindent 
{\Large \bf  Introduction} \bigskip

 In the study of vibrating mechanical systems, an important role is played by almost periodic functions, cf. e.g \cite{Dieud}, section I.3 for a concise presentation. By Bochner's criterion, a function f defined on the real line 
$\R$ with values in a complete metric space $X$ is almost periodic if, and only if the set of translates $f_a= f(a+.)$ is  precompact in the space of bounded continuous functions $C_b(\R, X).$  The closure $H(f)$ of this precompact set in $C_b(\R, X) $ is called the ``hull" of $f$. Any almost periodic function is bounded and uniformly continuous. A continuous periodic function is in particular almost periodic with a compact set of translates. The space of almost periodic functions is stable by many operations, for example a sum of any number of real valued continuous functions with arbitrary periods is almost periodic, while the sum of two non-constant periodic continuous functions with incommensurable periods is never periodic, cf.e.g.\cite{Dieud}, lemma I.3.1.7 p. 1.3.1.7. An important property coming from uniform convergence on $\R$ is that \begin{equation}g\in H(f) \Leftrightarrow f\in H(g).\end{equation} This property implies, among other things, that if $f$ has a limit at $+\infty$, it is actually a constant function. In the case of periodic functions, it is clear that either they are constant, or they oscillate around their mean-value or average. This global property is also important in the case of general almost periodic functions, and for this it is necessary to define their average.  It is known  that when $X$ is a Banach space,  there is a such an average, denoted by $\mu(f)$ such that 
\begin{equation}\label{moy}\lim_{t\rightarrow\infty} \sup_{a\in \R} ||\frac{1}{t} \int_a^{a+t} f(s) ds -\mu(f) ||ds = 0.\end{equation}More generally, when $X$ is a real Hilbert space, any almost periodic function with values in X has a formal representation of the form 
\begin{equation} f(t) \sim \mu(f) + \sum _n[ \cos (\mu_nt) f_n+ \sin (\mu_nt)  g_n] .\end{equation}  where \begin{equation} \label {comp} f_n = 2 \lim_{t\rightarrow\infty}\frac{1}{t}  \int_0^{t} \cos (\mu_ns) f (s)ds, \quad g_n = 2 \lim_{t\rightarrow\infty}\frac{1}{t}  \int_0^{t} \sin (\mu_ns) f (s)ds \end{equation} 
and we have the generalized Parseval identity
 \begin{equation}\label {GP} || \mu(f)||^2 + \frac{1}{2} \sum _n [ || f_n||^2+ || g_n||^2] = \lim_{t\rightarrow\infty}\frac{1}{t}  \int_0^{t}||f(s)||^2 ds .\end{equation}
The general properties of almost periodic functions have been studied in quite a few reference books, for instance by Bochner and  Von Neumann \cite{BVN}, Besicovitch \cite{Bes}, Corduneanu \cite{Cord1}, Amerio and Prouse \cite{AP}, Yoshisawa \cite{Yosh} and Fink \cite{Fink}. 
 
 \bigskip\noindent Oscillations only make sense for real valued functions, and the question can always be reduced to the case of mean-value 0, since in mechanics we usually take the equilibrium as the origin of the frame. The oscillation length, roughly speaking the maximal distance between two successive zeroes,  is a very important question  since its determination  allows to predict at which pace a mechanical structure will be vibrating. 
 
  \bigskip\noindent This paper is intended to provide a synthetic view on scattered results which were never presented in this form in the literature and is  divided in 7 additional sections. Section 1 gives an upper bound on the oscillation length for periodic functions. In section 2, we establish the existence of a uniform oscillation lentgh for any almost periodic function with mean-value 0. Sections 3,4 and 5 are devoted to upper bounds of the oscillation length for some almost periodic functions which appear in the study of continuum mechanics, and Section 6 shows the existence of a finite oscillation length with respect to time in any open subset of some vibrating continuous media. We close the paper by investigating precisely the vibrations of rectangular membranes with fixed edge and recalling a long standing open problem in the case of a square membrane.

\section{The periodic case} A continuous function $f$ defined on $\R$ is called T-periodic if \begin{equation} \forall t\in \R, \quad  f(t+T) = f(t) .\end{equation} The integral of $f$, hence its mean-value, is the same on each interval of length $T$, as a consequence if the mean-value is $0$ the following property is obvious
\begin{proposition} Let $f$ be any real-valued T-periodic function with mean-value 0. Then either $f\equiv 0$, or  on each interval $J$  of length  $ |J|\ge T$,  $f$ takes both positive and negative values. \end{proposition} 
\begin{proof} By contradiction, let us assume that $f$ has a constant sign on $J$. Since $\int_J f(t) dt =0 $, this implies $f\equiv 0$. \end{proof} 
\noindent  A function $f$ defined on $\R$ is called $\tau$-antiperiodic periodic if \begin{equation} \forall t\in \R, \quad  f(t+\tau) = - f(t) .\end{equation} It is obvious that, in this case,  f is $2\tau$ periodic with  mean-value  $0$.  The following result then  implies the existence of zeroes on any interval larger than the antiperiod.
\begin{proposition} Let $f$ be any real-valued $\tau$-antiperiodic function. Then  $f$ cannot keep a constant sign on any interval $J$  of length  $ |J|> \tau$,  \end{proposition} 
\begin{proof} By contradiction, let us assume that $f$ has a constant sign on $J = [a, a+ \tau + \varepsilon]$. By anti-periodicity, $f$ must take the opposite sign on $J = [a, a+ \varepsilon]$, which is impossible.  \end{proof} 
\begin{remark} On the other hand, $f$ can remain non-negative on any interval of length smaller than the period $T = 2\tau$. For instance we can have  $f =0$ on $[0, \tau-\varepsilon]$,  $f>0$ on $[\tau-\varepsilon, \tau]$ and $ f = 0$ on $[\tau,T-\varepsilon], $ in which case the anti-periodic function $f$ is nonnegative throughout $[0,T-\varepsilon].$
\end{remark}
\begin{remark} It may very well happen that the oscillation length of a T-periodic function be smaller than $T$. For instance the function $f(t) = \sin^2 t \sin3t $ is periodic with period $2\pi$ and its oscillation length is $\frac{\pi}{3}$.
\end{remark}

\section{A general theorem} \begin{thm} Let $f$ be any real-valued almost periodic function with mean-value 0. Then either $f\equiv 0$, or there exists an number $L = L(f)$ such that on each interval $J$  of length  $ |J|\ge L$,  $f$ takes both positive and negative values. \end{thm} 

The proof of the theorem relies on the following simple property.

\begin{lemma} Let $g$ be any real-valued almost periodic function with mean-value 0. If $g$ is non-negative throughout $\R^+$,  then  $g\equiv 0$. \end{lemma}

\begin{proof}If $g$ vanishes identically on $\R^+$, then $g\equiv 0$. By contradiction, let us assume that $g(a)>0 $  for some $a\in \R^+$. Then, by continuity there is $\delta>0$ for which 

$$\forall t\in [a, a+\delta], \quad g(t)\ge \frac{g(a)}{2} =: 2 \varepsilon $$ By almost periodicity, there exists a number $A>\delta >0$ such that for each $b\in \R$, the interval $J_b= [b, A+b]$ contains a point $c(b)$ such that $$ \sup_\R |g(s+c(b))-g(s)| \le \varepsilon. $$ Let us choose $b_n= n A, n\in \N$  and $ d_n := c(b_n) $, this provides
$$\forall t\in [a+d_n, a+ d_n+\delta], \quad g(t)\ge  \varepsilon $$  In particular we obtain 
$$\forall n\in\N, \quad \int _{a+nA} ^{a+nA+2A} g(t) dt\ge \int _{a+d_n} ^{a+d_n+\delta} g(t) dt  \ge  \varepsilon \delta $$ and finally 
\begin{equation}\label{>}\forall p\in\N, \quad \int _{a} ^{a+2pA} g(t) dt \ge  p \varepsilon \delta , \end{equation}
But \eqref{>} is incompatible with the mean value 0 which implies  \begin{equation}\lim_{p\rightarrow\infty} \frac{1}{2pA}\int _{a} ^{a+2pA} g(t) dt = 0.\end{equation}
\end{proof}

We now give the proof of the theorem

\begin{proof} By contradiction, let us assume that $f$ has a constant sign on some intervals  $J_n = [a_n, a_n+L_n] $  with $$ \lim_{n\rightarrow\infty } L_n = +\infty$$ 
By considering a subsequence, we may assume that the sign is the same for all n. By precompactness, a subsequence of the sequence of functions $g_n:= f(a_n+.)$ converges uniformly on $\R$ to a function $g\in H(f)$  which has a constant sign on $\R^+$. Then the lemma applied to $\pm g$ gives $g= 0$, and then $f\in H(g)$ implies that $f=0$. \end{proof}

\section{Gaps and oscillation} When dealing with numerical almost periodic functions, the first case to be considered is a sum of periodic functions with mean-value 0. In \cite{HK} it was shown that a finite sum of such periodic functions has an oscillation length bounded by the sum of the periods. This result can be used to find a uniform oscillation time for an infinite sum when the sum of the periods is finite. Unfortunately, it cannot be used even in the case of vibrating strings for which a very different method had to be devised in \cite{CH1, CH2}. Yet another method can be applied for anharmonic Fourier series when a uniform gap separates the exponents. This is the object of the next result, relying on a technique borrowed from Ingham \cite{Ing} and improving slightly a previous one which only appeared in an old preprint of the author dating back to 1988 and inside a chapter of the lecture note \cite{H2}  . 

\begin{thm} \label{gaposc} Let $f\not \equiv 0$ be a real-valued almost periodic function with mean-value 0 defined as the limit of the uniformly convergent series
\begin{equation} f(t) =  \sum _n[ a_n \cos (\mu_nt) + b_n \sin (\mu_nt) ]. \end{equation}  We assume that for some $ \gamma >0$ we have 
\begin{equation}\mu_0\ge \gamma ; \quad \forall n\in\N,\quad \mu_{n+1}-\mu_n\ge \gamma. \end{equation} Then on each interval $J$  of length  $ |J|\ge L: =\frac{3\pi}{\gamma} $,  $f$ takes both positive and negative values. \end{thm} 
\begin{proof} By performing an adequate time translation followed by a time scaling, it is sufficient to prove the result when $J= [-\pi,+\pi]$  and $ \gamma = \frac{3}{2 }$,  but to allow a possible future improvement, until the last step of the proof  we shall keep $\gamma\ge 1 $ arbitrary.  There no chance of concluding if $\gamma<1$, since the oscillation property would become false  even in the $2\pi$-periodic case. In the calculations, we shall rely repeatedly on the following elementary formula used by Ingham in \cite{Ing}:
\begin{equation} \label{Ing} \forall \lambda\not=\frac{1}{2} , \quad \int_J \cos(\frac{t}{2})\cos(\lambda t) dt = \frac {4\cos(\pi\lambda)}{1-4\lambda^2}\end{equation}
We proceed now by contradiction by assuming $f\ge 0$ on $J$. Because $f\not \equiv 0$, if all the $a_n$ are 0, the result follows because then $f$ is odd, it will either change sign on $J $ or be identically $0$  on $J$,  and by Ingham's Theorem, if $f \equiv 0$ on J, it implies $f \equiv 0$  on $\R$. Moreover, since the generalized coefficients $a_n$ tend to $0$ at infinity as a consequence of \eqref{GP}, we can introduce an index $j$ such that $|a_j|$ is equal to the maximum of $|a_n|. $ In fact, since positivity is preserved by a positive scaling, we can reduce the situation to the case where 
$$ |a_j| =1;\quad \forall n\in \N, |a_n|\le 1$$ where.  Now since $J= [-\pi,+\pi]$ and the sine function is odd, the condition $f\ge 0$ on $J$ implies 
$$  \int_J f(t) \cos(\frac{t}{2})(1\pm\cos(\mu_j t) dt \ge 0.$$ By adapting the $\pm$ sign (negative if $a_j =1$ and positive if $a_j=-1$), we deduce by expansion of the cosine terms in $f$  the inequality 
\begin{equation} \int_J  \cos(\frac{t}{2})\cos^2(\mu_j t) dt \le \sum_n  a_n \int_J \cos(\frac{t}{2})\cos(\mu_nt) dt \pm \sum_{n\not=j}  a_n \int_J \cos(\frac{t}{2})\cos(\mu_nt) \cos(\mu_jt) dt\end{equation} We shall compute all the terms by means of formula \eqref{Ing}. The RHS is  $$ \int_J  \cos(\frac{t}{2})\cos^2(\mu_j t) dt = \frac{1}{2}  \int_J  \cos(\frac{t}{2})[1+\cos(2\mu_j t)] dt = 2 \left[1+ \frac{\cos(2\pi\mu_j)}{1-16\mu_j^2}\right]$$ It follows easily, since $\mu_j\ge\gamma$, that 
$$ \int_J  \cos(\frac{t}{2})\cos^2(\mu_j t) dt \ge 2\left[1- \frac{1}{16\gamma^2-1}\right] = 2\left(\frac{16\gamma^2-2}{16\gamma^2-1}\right)$$ Next we find 
$$ \sum_n  a_n \int_J \cos(\frac{t}{2})\cos(\mu_nt) dt \le \sum_n  \frac{4}{4(n+1)^2\gamma^2-1} $$ The function $$ S(\gamma) =  \sum_{p=1}^\infty   \frac{1}{4p^2\gamma^2-1}$$ can be computed and $$S(\gamma) = \frac{1}{2}- \frac{\pi}{4\gamma} \cot \frac{\pi}{2\gamma} $$  To appraise the last infinite sum, we use the linearization formula 
$$ \cos\mu_nt \cos\mu_j t= \frac{1}{2}\left[\cos(\mu_n+\mu_j) t+ \cos(\mu_n-\mu_j) t \right] $$ and we find 
$$\int_J \cos(\frac{t}{2})\cos(\mu_n t) \cos(\mu_j t) dt = 2\left(\frac{\cos\pi(\mu_n+\mu_j)}{4(\mu_n+\mu_j)^2-1} + \frac{\cos\pi(\mu_n-\mu_j)}{4(\mu_n-\mu_j)^2-1}\right)$$ therefore, after observing that  $ \mu_n+\mu_j\ge 3\gamma$ whenever $n\not=j$, we obtain 
$$ \pm \sum_{n\not=j}  a_n \int_J \cos(\frac{t}{2})\cos(\mu_nt) \cos(\mu_jt) dt \le 2 \left(1+\delta)(3S(\gamma)-\frac{1}{4\gamma^2 -1}- \frac{1}{16\gamma^2 -1} \right) . $$ By combining all the terms we find the inequality $$ \frac{16\gamma^2-2}{16\gamma^2-1} \le 5 S(\gamma) - \frac{1}{4\gamma^2 -1}- \frac{1}{16\gamma^2 -1} $$ For $\gamma =\frac{3}{2}$ we obtain $$ \frac{34}{35} \le 5 S(\tfrac{3}{2})-\frac{1}{8} -\frac{1}{80} \Longleftrightarrow \frac{621}{560}\le 5 S(\tfrac{3}{2}) $$ which is false since $$S(\tfrac{3}{2}) = \frac{1}{2} -\frac{\pi}{6} \cot \frac{\pi}{3} = \frac{1}{2} -\frac{\pi}{6\sqrt 3}  < \frac{1}{5}$$ and this concludes the proof. 
 \end{proof} 

\begin{remark} This result improves Proposition 3  of {\rm \cite{H2}}, reducing the condition on $|J|$ by a factor $\frac{3}{4}$, but it is not optimal and we still do not know if the weaker condition  $ |J| \ge \frac {2\pi}{\gamma} $ is sufficient. By comparing with the situation described in {\rm\cite{Ing}}, one would expect that the limiting case $ |J| = \frac {2\pi}{\gamma} $ might be excluded, which is problematic for the study of vibrating strings considered in the next section. \end{remark}

\section{Vibrating strings} Let us consider the simple equation 

\begin{equation} \label{string} \begin{cases}
u_{tt}-u_{xx} + a^2 u = 0 \quad  \rm{on}\,\,  \R\times(0, l)   \\
u(t, 0) = u(t, l) = 0 \quad  \rm{on}\,\,  \R 
\end{cases}\end{equation} where $l>0, a\ge 0$. 
In this case, for any $(u_0, v_0) \in H_0^1(0, l) \times L^2(0, l) $, there is a unique solution $ u\in C(\R, H_0^1(0, l))\cap C^1(\R, L^2(0, l))\cap C^2(\R, H^{-1} (0, l))$ of \eqref{string} such that $u(0) = u_0\,\,  \rm{and}\,\, u_t(0)= v_0 $, given by the formula 
\begin{equation} u(t, x) =  \sum _n[ a_n \cos (\mu_nt) + b_n \sin (\mu_nt) ] \sin (\tfrac{\pi}{l} nx).\end{equation} with 
$$ u_0(x) = \sum _n  a_n \sin (\tfrac{\pi}{l} nx) ; \quad v_0(x) =  \sum _n \mu_n  b_n  \sin (\tfrac{\pi}{l} nx) $$ and $$ \mu_n := \sqrt{(\tfrac{\pi}{l})^2n^2+a^2}$$ The application of Theorem \ref{gaposc} shows that whenever $\frac{x}{l}$ is irrational and $u\not\equiv 0$, the function $t\longrightarrow u(t, x)$ takes both positive and negative values on the real interval $J$ as soon as $$ |J|\ge l \left(\sqrt{1+ a^2(\tfrac{l}{\pi})^2}  + \sqrt{4+ a^2(\tfrac{l}{\pi})^2}\right) $$ When $a=0$ we find $$ |J|\ge 3l $$ when the correct condition is $ |J|\ge 2l $. Moreover in \cite{CH1, CH2, OscPort} it was shown that the correct value is also $\le 2l$ even for $a>0$.  It would be nice to improve the method of Theorem \ref{gaposc} in order to recover at least the result of \cite{CH1, CH2, OscPort} under the strict condition $ |J|> 2l$.

\section{Vibrating membranes} The oscillation problem for vibrating membranes driven by the equation \begin{equation} \label{memb} \begin{cases}
u_{tt}-\Delta u= 0 \quad  \rm{on}\,\,  \R\times\Omega   \\
u(t, x, y) = 0  \quad  \rm{on}\,\,  \R \times\partial \Omega
\end{cases}\end{equation} where $\Omega$ is an open subset of $\R^2$ is much more complex, even for simple domains as a square. As was established in \cite {HK}, the oscillation time of the function   $t\longrightarrow u(t, x, y)$  with $(x, y)$ fixed can be arbitrarily long even for analytic solutions with a finite number of harmonics. A general method relying on the positivity of the first eigenfunction of the Laplacian, associated with the eigenvalue $\lambda_1>0$, shows that the solution cannot remain positive on $J\times \Omega$ as soon as 
$|J| > \frac{\pi}{\sqrt{\lambda_1}}$, and this result is optimal for initial data of the form $u(0) = \phi; u_t(0) = 0$ with  $\phi$ an eigenfunction relative to $\lambda_1.$ So it is natural to ask about oscillations in an open subset $\omega\subset \Omega. $ If $\Omega$ is a rectangle, It turns out that the oscillation time of all solutions inside $\omega$ has a universal bound as long as $\omega$ contains the intersection of strip parallel to one side with $\Omega$, and the bound does not depend on the thickness of the strip. We give the details in Section 7. But for all other subsets $\omega$, the existence of such a bound is unknown. At least the next result shows that no solution can remain of a fixed sign on $\R^+\times \omega$. 

\section{A general property}
In this section, we consider the general second order evolution equation 
\begin{equation} \label{WE}  u''+ Au(t) = 0 \end{equation}  where A is a coercive positive self-adjoint operator with compact inverse on the  real Hilbert space $H = L^2(\Omega)$ and $V = D(A^{1/2})$ with the usual norm, $\Omega $ being an open bounded domain of $\R^N$. It is well known that all mild solutions of the equation are almost periodic : $\R \rightarrow V$. Given a non empty open subset $\omega\subset\Omega$, we have the following general result: 
\begin{thm} \label{gen} Let $u\not \equiv 0$ be any mild solution  of equation \eqref{WE}.  Then, assuming that the eigenfunctions of $A$ are analytic in  $\Omega$, there is $L>0$ such that $u$ takes both positive and negative values in the cylinder $J\times \omega$ whenever $|J|\ge L$. \end{thm} \begin{proof} By contradiction. The function $f(t):= \int_\omega u(t,x) dx$ is real valued and almost periodic with mean-value 0. Assuming that $u$ remains nonnegative on $J_n\times \omega$ with $|J_n|$ tending to infinity, we end up with a solution $v \in H(u)$ which remains nonnegative on $\R^+\times \omega.$ Then $g(t):= \int_\omega v(t,x) dx$ is almost periodic and nonnegative with mean-value 0 , which implies $g\equiv 0 $ on $\R.$ This implies that $v$ vanishes throughout  $\R^+\times \omega$, which in turn, by the analyticity  hypothesis applied to the time Fourier  components of $v$ given by formula \eqref{comp} implies $v\equiv 0$ and finally $u\equiv 0$.   \end{proof} \begin{corollary} Let $u\not \equiv 0$ be any mild solution of the equation \eqref{memb}.  Then, for any nonempty  open subset $\omega\subset \Omega,$ there is $L>0$ such that $u$ takes both positive and negative values in the cylinder $J\times \omega$ whenever $|J|\ge L$. \end{corollary}

\section{Back to vibrating membranes} In this section, we explain precisely what happens in a rectangle and we formulate two open questions.  We consider a rectangle $\Omega = (0, \pi)\times (0, L)$ and the horizontal strip $\omega = (\alpha, \beta)\times (0, L) $ where $0\le\alpha<\beta\le \pi$. Assuming that a solution $u$ of \eqref{memb} remains nonnegative in $J\times\omega$, we consider the function $$w(t, x): = \int _0^L u(t, x, y) \sin (\gamma y) dy $$ where $$ \gamma : = \frac{\pi}{L}.  $$ Therefore $$w_{tt} (t, x): = \int _0^L u_{tt}(t, x, y) \sin (\gamma y) dy $$
Therefore $$w_{xx} (t, x): = \int _0^L u_{xx}(t, x, y) \sin (\gamma y) dy = \int _0^L (\Delta ut, x, y)- u_{yy}(t, x, y) \sin (\gamma y) dy$$
$$= \int _0^L [\Delta u(t, x, y)- \gamma^2u(t, x, y)] \sin (\gamma y) dy = \int _0^L \Delta u(t, x, y)\sin (\gamma y) dy- \gamma^2 w (t x) $$ hence w satisfies the equation 
\begin{equation} \label{string2} \begin{cases}
w_{tt}-w_{xx} + \gamma^2 w = 0 \quad  \rm{on}\,\,  \R\times(0, l)   \\
w(t, 0) = w(t, \pi) = 0 \quad  \rm{on}\,\,  \R 
\end{cases}\end{equation}  
Now $$ \forall (t,x,y)\in J\times\omega, \,\,u(t,x,y) \ge 0 \Longrightarrow \forall (t, x)\in J\times (\alpha, \beta),\,\,  u(t,x)\ge 0. $$  By relying for instance on \cite{OscPort} we obtain \begin{proposition} Assuming that a solution $u$ of \eqref{memb}  with $\Omega = (0, \pi)\times (0, L)$ remains nonnegative in $J\times\omega$ with $\omega = (\alpha, \beta)\times (0, L) $ and $0\le\alpha<\beta\le \pi$, we have either  $ u\equiv 0 $ or $ |J| < 2\pi .$ \end{proposition}
\begin{remark} In case $\alpha = 0 $ and $\beta=L$, the condition is reinforced and becomes $$ |J| \le \frac{\pi} { \sqrt {1+ \gamma^2 }}.$$ which represents a jump by a factor larger than 2 which tends to infinity with $\gamma.$ We have no idea of what could be the best estimate in general as a function of $(\alpha, \beta)$.\end{remark}
\begin{remark} In the special case $\Omega = (0,\pi)\times (0,\pi)$, and  $\omega = (0, \frac{\pi}{2})\times (0, \frac{\pi}{2}) $, we know by {\rm Theorem \ref{gen}} that any solution $u$ of \eqref{memb} has a uniform oscillation time on $\omega$, but it is still unknown whether there is an oscillation time independent of $u$, a problem raised in 1985 and already mentioned in {\rm \cite{Op}}.   \end{remark}

\end{document}